\documentclass{amsart}

\usepackage[utf8]{inputenc}
\usepackage[T1]{fontenc}
\usepackage{a4wide, amsmath, amssymb, mathtools, tabularx, amsthm}

\newtheorem{theorem}{Theorem}[section]
\newtheorem{lemma}[theorem]{Lemma}

\newtheorem{corollary}[theorem]{Corollary}
\newtheorem{conjecture}[theorem]{Conjecture}

\newtheorem{definition}[theorem]{Definition}
\newtheorem{remark}[theorem]{Remark}

\numberwithin{equation}{section}

\DeclareMathOperator{\lcm}{lcm}
\DeclareMathOperator{\ord}{ord}

\title[A primality test for $Kp^n+1$ numbers]{A primality test for $Kp^n+1$ numbers and a generalization of Safe Primes and Sophie Germain Primes}
\date{}
\author{A. Ramzy}
\address{Department of Mathematics, Faculty of Education, Azhar University, Cairo, Egypt}
\email{hosam7101996@gmail.com}
\begin{document}

\begin{abstract}
In this paper, we provide a generalization of Proth's theorem for integers of the form $Kp^n+1$. In particular, a primality test that requires only one modular exponentiation similar to that of Fermat's test without the computation of any GCD's. We also provide two tests to increase the chances of proving the primality of $Kp^n+1$ numbers (if they are primes indeed). As a corollaries of these tests we provide three families of integers $N$ whose primality can be certified only by proving that $a^{N-1} \equiv 1 \pmod N$ (Fermat's test). We also generalize Safe Primes and define those generalized numbers as $a$-SafePrimes for being similar to SafePrimes (since $N-1$ for these numbers has large prime factor the same as SafePrimes), we address some questions regarding the distribution of those numbers and provide a conjecture about the distribution of their related numbers $a$-SophieGermainPrimes which seems to be true even if we are dealing with $100$, $1000$, or $10000$ digits primes.
\end{abstract}
\maketitle

\section{INTRODUCTION}

One of the fundamental facts about prime numbers is Fermat's result that if $N$ is prime then for every integer $a$ we have

\begin{equation}
	\label{fer1}
	a^N \equiv a \pmod N.
\end{equation}

This leads to the well-known Fermat's primality test which is a probabilistic test, 
but since the verification of \eqref{fer1} for a given $a$ and $N$ is computationally inexpensive, even for large $N$, there has been a number of results studying additional conditions that could be combined with Fermat's test to conclude primality. For example, in 1878 Proth (see \cite{Rib}) presented the following theorem to determine whether $N$ is prime or not when $N$ is of the form $N=K2^n+1$ (Proth numbers)

\begin{theorem} [Proth, 1878]
	Let $N=K2^n+1$, where $K$ is odd and $K<2^n$. If there exists an integer $a>1$, such that $a^{\frac{N-1}{2}} \equiv -1 \pmod {N}$, then $N$ is prime.
\end{theorem}

In 1914 an important and useful result was proved by Pocklington, which only requires a partial factorization of $(N-1)$ 

\begin{theorem}[Pocklington, 1914]\label{Pock1}
	Let $N-1=q^nR$ where $q$ is a prime, $n \ge 1$, and $q \nmid R$. Assume that there exists an integer $a>1$ such that:
	\begin{enumerate}
		\item[(i)] $ a^{N-1} \equiv 1 \pmod N$, and
		\item[(ii)] gcd$(a^{\frac{N-1}{q}}-1, N)=1$.
	\end{enumerate}
	Then each prime factor of $N$ is of the form $mq^n+1$, with $m \ge 1$.
\end{theorem}


As a result of \ref{Pock1} he gave the first generalization of Proth's theorem suitable for numbers of the form $N=Kp^n+1$, which are called \emph{Generalized Proth numbers} (see \cite{Rib}): 
\begin{theorem}[Pocklington, 1914]\label{Pock3} Let $N=Kp^n+1$ with $K<p^n$, and $p$ is prime. If there exists an integer $a>1$ such that:
	\begin{enumerate}
		\item[(i)] $a^{N-1} \equiv 1 \pmod N$, and 
		\item[(ii)] gcd$(a^{\frac{N-1}{p}}-1, N)=1$,   
	\end{enumerate}
	then $N$ is prime.
\end{theorem}

In \cite{GrOMSa} the following simpler generalization was presented by Grau, Oller-Marcen, and Sadornil.

\begin{theorem}[Grau et al.,  2015]\label{GOS}
	Let $N=Kp^n+1$, where $p$ is prime and $K<p^n$. Assume that $a \in \mathbb{Z} $ is a $p$-th power non-residue, then $N$ is a prime if and only if $ \Phi_p(a^{\frac{N-1}{p}}) \equiv 0\pmod N $. 
\end{theorem} 

Then they gave a more useful test which increases the chance of proving the primality of $N$ (if $N$ is prime indeed). Note that we replaced $(J)$ in the original theorem by $(n-j)$ in order to compare it to theorem \ref{thm2} in the following section, namely, the conditions $n-1 \geq j \geq 0$ and $2(n-j) > \log_p(K)+n$ in the following theorem are equivalent to the conditions $1 \leq J \leq n$ and $2J > \log_p(K)+n$ in the original theorem in \cite{GrOMSa}. 

\begin{theorem}[Grau et al.,  2015]\label{thmjose}
	Let $N=Kp^n+1$, where $p$ is prime. If there exists $n-1 \geq j \geq 0$ such that:
	\begin{itemize}
		\item[(i)]$\Phi_p(a^{Kp^{n-j-1}}) \equiv 0\pmod N $
		\item[(ii)]$2(n-j) > \log_p(K)+n$.
	\end{itemize} 
	Then $N$ is prime.
\end{theorem} 

Theorem \ref{GOS} states that if $\Phi_p(a^{\frac{N-1}{p}}) \equiv 0\pmod N$ such that $N=Kp^n+1$, then $N$ is prime. In practice, $(\Phi_p(x))$ is easily computed as $\frac{x^p-1}{x-1}$. Thus, to use Theorem \ref{GOS} we would need to verify that $a^{N-1} \equiv 1 \pmod N$ and that $a^{\frac{N-1}{p}}-1$ is invertible $\pmod N$. The latter condition can be verified by checking that gcd$(a^{\frac{N-1}{p}}-1,N)=1$, which is well known to be an $O(\log N)$ computation, the same can be said as for computing $\Phi_p(a^{Kp^{n-j-1}})$ in \ref{thmjose} since it requires the gcd step. Note that we can compute $\Phi_p(x)$ as $1+x+x^2+\cdots+x^{p-1}$ but only when $p$ is small, since computing $\Phi_p(x)$ that way would be time consuming if $p$ is large.

We organized the paper as follows. In Section 2 we prove a couple of simple lemmas that we use to prove (Theorem \ref{main}) as well as a useful generalization of it (Theorem \ref{thm2}), we also put a Remark at the end of the Section to clarify the main contributions of Theorems \ref{main} and \ref{thm2} and provide a simple comparisons to existing methods. In Section 3 we provide two more useful tests but only for integers of the form $Kp^n+1$ such that  $p^j \ge p^{n-j} \ge (N-1)^{1/3}$ or $p^j \ge p^{2(n-j)} $ and $p^{n-j} \ge (N-1)^{2/7}$, and after each theorem we put a simple comparisons to existing methods, note that the above conditions on $p^{n-j}$ will be satisfied mostly, since in most cases we will have $p^n$ be much larger than $K$. Also as a corollaries we give two infinite families of integers $N$ whose primality can be certified only by proving that $a^{N-1} \equiv 1 \pmod N$. In section 4 we give a family of integers similar to Safe Primes (since $N-1$ for these numbers has large prime factor the same as Safe primes), those numbers will be defined as $a$-SafePrimes, we address some questions regarding the distribution of those classes of primes and provide some computations for $2$-SafePrimes, and the distribution of the related numbers $2$-SophieGermainPrimes will be addressed in section 5.
In Section 5 we provide some estimates to find the probability that a given prime is $a$-SophieGermainPrime, then we give some computations for $2$-SophieGermainPrimes  which shows that the accuracy of the estimates is very acceptable, even if we are dealing with $100$, $1000$ or $10000$ digits primes (random or consecutive), we also give a conjecture about the distribution of $2$-SophieGermainPrimes (and $a$-SophieGermainPrimes in general) which helped us to discover many primes larger than $10^{999}$ and some of them were larger than $10^{9999}$.


\section{A Generalization of Proth's Theorem} 

In this section we shall state and prove theorems \ref{main} and \ref{thm2}, whose provide a simple primality test for generalized Proth's numbers $N=Kp^n+1$. To prove the theorems we require two  lemmas.

\begin{lemma}\label{lem1}
	Assume that $A, P$ are integers with  $1\leq A \leq P$. If there is an integer $D>0$ such that 
	\[
	\frac{AP+1}{DP+1} \in \mathbb{Z},
	\]
	then we must have $D=A$.
\end{lemma}

\begin{proof}
	If $DP+1 \mid AP+1$ we must have $D\leq A$. Write $A=cD+r$ with $c> 0$ and $0\leq r<D$. Hence
	
	\begin{align*}
		\frac{(cD+r)P+1}{DP+1}&=\frac{cDP+rP+1+c-c}{DP+1}\\
		&=c+\frac{rP+1-c}{DP+1}.
	\end{align*}

	Note that $c\leq A \leq P$, so if $r\geq 1$ we will have $rP-c+1\geq 1$, hence the numerator of the last fraction is a positive number which is strictly less than its denominator (since $r<D$). Thus the fraction can only be an integer when $r=0$ and $c=1$, proving the result.
\end{proof}

Recall that the \emph{order} of $a \pmod N $ is the least positive integer $m$ such that $a^m \equiv 1 \pmod N$. We shall denote the order of $a \pmod N $ with $\ord_N(a)$.

\begin{lemma}\label{lem2}
	Assume that $p^n \mid \ord_N(a)$ where $p$ is prime and gcd$(p,N)=1$, then there must exist a prime divisor $q$ of $N$ such that $p^n \mid \ord_q(a)$, therefore $p^n \mid q-1$ .
\end{lemma}

\begin{proof}
	Assume that the prime factorization of $N=q_1^{e_1} \cdots q_s^{e_s}$, then $\ord_N(a)=\lcm(\ord_{q_1^{e_1}}(a),\cdots,\ord_{q_s^{e_s}}(a))$, but since $p^n \mid \ord_N(a)$, then there must exist a prime divisor $q_c$ of $N$ such that $p^n \mid \ord_{q_c^{e_c}}(a)$, but since $\ord_{q_c^{e_c}}(a) = q_c^k \times \ord_{q_c}(a) $, where $k \leq e_c-1$, and since gcd$(p,N)=1$, then $p^n$ must divide $\ord_{q_c}(a)$, therefore, $p^n \mid q_c-1$ .
\end{proof}

\begin{theorem}\label{main}
	Let $N=Kp^n+1$, where $p$ is prime, $p^n \ge K$. Assume that there exists an integer $a$ such that:
	\begin{enumerate}
		\item[(i)] $a^{Kp^{n-1}} \equiv L \neq 1 \pmod N$, and
		\item[(ii)] $L^p \equiv 1 \pmod N$.
	\end{enumerate} 
	Then $N$ is prime.
\end{theorem} 

\begin{proof}
	Assume both of (i) and (ii) hold, then $p^n \mid \ord_N(a)$, but since gcd$(p,N)=1$, then there must exist a prime divisor $q$ of $N$ such that $p^n \mid q-1$, due to Lemma \ref{lem2}, hence $q=hp^n+1$ divides $N$ implying  $\frac{Kp^n+1}{hp^n+1}\in \mathbb{Z}$, but due to Lemma \ref{lem1} we must have $h=K$, therefore, $q=N$ and $N$ is prime.
\end{proof}

As corollary, we show that the primality of $2Kp+1$ numbers with $2K < \log_{a}(2Kp+1)$ can be certified only by proving that $a^{N-1} \equiv 1 \pmod N$, and then we generalize \ref{main}.

\begin{corollary}\label{cor24} Let $N=2Kp+1$, where $p$ is prime, $2K \leq p$, and $2K < \log_{a}(2Kp+1)$. Then $N$ is prime if and only if
	\begin{equation} 
		a^{N-1} \equiv 1 \pmod N 
	\end{equation}
	
\end{corollary}

\begin{proof}
	By the hypothesis we have $2K < \log_{a}(2Kp+1)$, hence $a^{2K} \equiv L \neq 1 \pmod N $, and $a^{N-1} \equiv 1 \pmod N$. Thus all the conditions of Theorem \ref{main} are satisfied and $N$ is prime. The converse is simply Fermat's primality test.
\end{proof}

\begin{theorem}\label{thm2} Let $N=Kp^n+1$, where $p$ is prime, $p^{n-j} \ge Kp^{j} \Longrightarrow p^{2(n-j)} \ge Kp^n \Longrightarrow 2(n-j) \geq \log_{p}(K)+n$ where $(0 \leq j \leq n-1)$. Assume that there exists an integer $a>1$ such that: 
	\begin{enumerate}
		\item[(i)] $a^{Kp^{n-j-1}} \equiv L \neq 1 \pmod N$, and
		\item[(ii)] $L^{p^{j+1}} \equiv 1 \pmod N$.
	\end{enumerate}
	Then $N$ is prime. 
\end{theorem}

\begin{proof}
	As in the proof of Theorem \ref{main} we can deduce that there must exist a prime divisor $q$ of $N$ such that  $p^{n-j} \mid q-1$,  hence $q=hp^{n-j}+1$ divides $N$, which means that  $\frac{Kp^{j}.p^{n-j}+1}{hp^{n-j}+1}\in \mathbb{Z}$, but due to Lemma \ref{lem1} we must have $h=Kp^{j}$, therefore, $q=N$ and $N$ is prime.
\end{proof}

\begin{remark}
	The main contribution of Theorem \ref{main} is to show that the gcd step in \ref{GOS} can in fact be replaced by simply verifying that $a^{\frac{N-1}{p}} \neq 1 \pmod N$. Computationally and theoretically this provides a simpler generalization of Proth's theorem. Note that for small values of $p$ we can compute $\Phi_p(X)$ as $1+X+X^2+\cdots +X^{p-1}$ and then theorem \ref{GOS} will not require any gcd computations and will be equivalent to \ref{main} (computationally).
	
	As an example to illustrate Theorem \ref{main}, consider $N=2\cdot107^3+1=2450087$. We can verify that (i) $2^{2\cdot107^2} \equiv 1302367 \pmod {2450087}$ and (ii)$1302367^{107} \equiv 1 \pmod {2450087}$, and thus the hypothesis of Theorem \ref{main} are satisfied and we deduce that $N$ is prime. The same can of course be obtained from Theorem \ref{GOS} by computing $\Phi_{107}(2^{2\cdot107^2})\pmod {2450087}$, but the latter requires the additional step of verifying that gcd$(1302367-1, 2450087)=1$ (since $p=107$ is relatively big) in order have a meaningful ``denominator" when evaluating $\phi_{107}$ $\pmod {2450087}$. 
	
	The same can be said as for the comparison of \ref{thm2} to \ref{thmjose}, since if $p$ is big, then \ref{thmjose} will require the gcd step, since then it will be difficult to compute $\Phi_p(a^{Kp^{n-j-1}})$ as $1+(a^{Kp^{n-j-1}})+(a^{Kp^{n-j-1}})^2+ \cdots + (a^{Kp^{n-j-1}})^{p-1}$. Note that if $N$ is indeed prime, then the chances of proving the primality of $N$ are even whether we used \ref{thmjose} or \ref{thm2}, since if $N=Kp^n+1$ then the base ($a$) will fail to satisfy condition (i) in theorems \ref{thm2} and \ref{thmjose} only when ($a$) itself is a $p^{j+1}$-th power residue modulo $N$, which happens exactly for $\frac{1}{p^{j+1}}$ of the possible choices for ($a$), which means that condition (i) in \ref{thm2} and \ref{thmjose} will be satisfied by a number of bases that equals to $Kp^n-Kp^{n-j-1}$, but since ($j$) in \ref{thm2} can be as large as ($j$) in \ref{thmjose} then the number of the satisfying bases of (i) in both of \ref{thm2} and \ref{thmjose} is the same.

\end{remark}

\section{Increasing the chances of proving the primality of $Kp^n+1$ primes}
In this section we provide two tests for $Kp^n+1$ numbers such that $p^j \ge p^{n-j} \ge (N-1)^{1/3}$ or $p^j \ge p^{2(n-j)} $ and $p^{n-j} \ge (N-1)^{2/7}$, those two tests can increase the chances of proving the primality of $Kp^n+1$ primes by increasing the number of the satisfying bases ($a$) of their conditions of primality. We shall give an example after each theorem to show their main contributions.

\begin{theorem}\label{thm31}
	Let $N=Kp^n+1 \neq p^{3(n-j)}+1$,  where $p$ is prime and $p^j \ge p^{n-j} \ge (N-1)^{1/3}$. If there exists an integer $a>1$ such that: 
	\begin{enumerate}
		\item[(i)] $a^{Kp^{n-j-1}} \equiv L \neq 1 \pmod N$, and
		\item[(ii)] $L^{p^{j+1}} \equiv 1 \pmod N$.
	\end{enumerate}
	Then $N$ is prime.
\end{theorem}

\begin{proof}
	Assume both of (i) and (ii) hold, then $p^{n-j} \mid \ord_N(a)$, but since $gcd(p,N)=1$, then there must exist a prime divisor $q$ of $N$ such that $p^{n-j} \mid q-1$, due to Lemma \ref{lem2}, therefore, we have two cases:
	\begin{enumerate}
		\item[(i)] $q=N$ and $N$ is prime, or
		\item[(ii)] $q < N$ and $N$ is composite.
	\end{enumerate}
	Now, we will try to exclude (ii) by contradiction. In(ii) we have $q \mid N$, hence, $\frac{N}{q}=M \in \mathbb{Z}$, but since $q=ap^{n-j}+1$ then we must have $M=bp^{n-j}+1$ as well, now, if we multiply $q$ by $M$ we will get $N=Kp^n+1=abp^{2(n-j)}+(a+b)p^{n-j}+1 \Longrightarrow Kp^{j}=abp^{n-j}+a+b$, therefore we must have $p^{n-j} \mid (a+b) \Longrightarrow a+b \ge p^{n-j} \Longrightarrow a+b \ge (N-1)^{1/3}$, but since $ab$ can't exceed or equal to $(N-1)^{1/3}$, and since $ab \ge a+b-1$, then the only chance we can have $ab < (N-1)^{1/3}$ and $p^{n-j} \mid (a+b)$ is when $p^{n-j}=(N-1)^{1/3}$, $ab=p^{n-j}-1$ and $a+b=p^{n-j} \Longrightarrow N=p^{3(n-j)}+1$, but since we excluded the latter case by the assumption on $N$, then we must have $q=N$ and $N$ is prime.
	
\end{proof}
To see the utility of theorem \ref{thm31} you should recall that if $N=Kp^n+1$ is prime, then $a$ will fail to satisfy condition (i) of Theorems \ref{thmjose} \ref{thm2} and \ref{thm31} only when $a$ itself is a $p^{j+1}$-th power residue modulo $N$, which happens exactly for $\frac{1}{p^{j+1}}$ of the possible choices for $a$, but $j$ in \ref{thm31} can be larger than $j$ in theorems \ref{thm2} and \ref{thmjose}. Hence (i) in \ref{thm31} is satisfied by more bases, which equals to $Kp^n-Kp^{n-j-1}$. To give a concrete example, consider $N=2\cdot 3^{17}+1$. Then according to theorems \ref{thmjose} and \ref{thm2} we can only have $j \leq 8$, however taking $1 \le j \le 8$ and $a=136837116$ will not satisfy (i) in neither \ref{thmjose} nor \ref{thm2}. But in \ref{thm31} we can have $j \leq 11$, and by taking $j=9$ for example, we obtain $136837116^{2\cdot 3^{7}}\equiv 216758952 \pmod N$ and the result follows from Theorem \ref{thm31}.

As a corollary, we give the following class of integers $N$ whose primality can be certified only by proving that $a^{N-1} \equiv 1 \pmod N$ (Fermat's test).

\begin{corollary}\label{cor32}
	Let $N=2Kp^2+1$,  where $p$ is prime, $2K < p$ and $2K < \log_{a}(2Kp^2+1)$ . If there exists an integer $a>1$ such that:
	\begin{equation}
		a^{N-1} \equiv 1 \pmod N.
	\end{equation}
	Then $N$ is prime.
\end{corollary}

\begin{proof}
	Since $2K<p$ then $j=1$ will satisfy the condition $p^j \ge p^{2-j} \ge (N-1)^{1/3}$, and since $2K < \log_{a}(2Kp^2+1)$ then we will have $a^{2Kp^{n-j-1}}=a^{2Kp^{2-1-1}}=a^{2K}  \equiv L \neq 1 \pmod N$, also we have $a^{N-1} \equiv 1 \pmod N$. Thus all the conditions of Theorem \ref{thm31} are satisfied and $N$ is prime.
\end{proof}
Note that $a=2$ is the best choice in \ref{cor32}, since then we will have many numbers $K$ satisfying the condition $2K < \log_{a}(2Kp^2+1)$. The same can be said as for \ref{cor24} and \ref{cor35}.

Now, we will proceed to the second test which requires the following lemma.

\begin{lemma}\label{lem33}
	If $X > 11+4\sqrt{7}$, $a \le \frac{X}{2}$, $a \neq 1$, $a \neq \sqrt{X+1}$ and $\frac{(X-a)(a)+1}{X}=M \in \mathbb{Z}$. then $M > \sqrt{X}$.  
\end{lemma}

\begin{proof}
	If $\frac{(X-a)(a)+1}{X} \in \mathbb{Z}$ then we should have $\frac{-a^2+1}{X} \in \mathbb{Z}$ as well, therefore we could have: 
	\begin{itemize}
		\item[(i)] $\frac{-a^2+1}{X}=0 \Longrightarrow a=1 \Longrightarrow M=1 $
		\item[(ii)]$\frac{-a^2+1}{X}=-1  \Longrightarrow a=\sqrt{X+1} \Longrightarrow M=\sqrt{X+1}-1 $
		\item[(iii)]$\frac{-a^2+1}{X} \le -2 \Longrightarrow  \frac{a^2-1}{X} \ge 2 $.
	\end{itemize}
	In (iii) we have $\frac{a^2-1}{X} \ge 2 $, but since $X > 11+4\sqrt{7}$ then we must have $ a > \sqrt{X}+2$, but for which $X > 11+4\sqrt{7}$ and $\frac{X}{2} \ge a > \sqrt{X}+2$ we will have $\frac{(X-a)(a)+1}{X} > \sqrt{X}$ which proves our result. 
\end{proof}

\begin{theorem}\label{thm34}
	Let $N=Kp^n+1 \neq p^{3(n-j)}+1$, $N \neq (\sqrt{p^{n-j}+1}-1)p^{3(n-j)}+1$ and $N \neq H^3p^{3(n-j)}+1$, where $p$ is prime, $p^j \ge p^{2(n-j)} $, $p^{n-j} \ge (N-1)^{2/7}$ and $p^{n-j} > 11+4\sqrt{7}$. If there exists an integer $a>1$ such that:
	\begin{enumerate}
		\item[(i)] $a^{Kp^{n-j-1}} \equiv L \neq 1 \pmod N$, and
		\item[(ii)] $L^{p^{j+1}} \equiv 1 \pmod N$.
	\end{enumerate}
	Then $N$ is prime.
\end{theorem}

\begin{proof}
	As in \ref{thm31}, we can deduce that there must exist a prime divisor $q$ of $N$ such that $p^{n-j} \mid q-1$, therefore, we have two cases:
	\begin{enumerate}
		\item[(i)] $q=N$ and $N$ is prime, or
		\item[(ii)] $q < N$ and $N$ is composite.
	\end{enumerate}
	Now, we will try to exclude (ii) by contradiction. In(ii) we have $q \mid N$, hence, $\frac{N}{q}=M \in \mathbb{Z}$, but since $q=ap^{n-j}+1$ then we must have $M=bp^{n-j}+1$ as well, now, if we multiply $q$ by $M$ we will get $N=Kp^n+1=abp^{2(n-j)}+(a+b)p^{n-j}+1 \Longrightarrow Kp^{j}=abp^{n-j}+a+b$. Therefore we must have $p^{n-j} \mid a+b$ and $p^{n-j} \mid (ab+\frac{a+b}{p^{n-j}})$ which leads us to the following two cases (let $p^{n-j}=X$ for abbreviation):
	\begin{enumerate}
		\item[(1)] $a,b <X \Longrightarrow a+b=X \Longrightarrow X \mid (ab+1) \Longrightarrow \frac{ab+1}{X}=\frac{N-1}{X^3}$
		\item[(2)] $a=AX+c$, $b=X-c \Longrightarrow  X \mid (AX^2 - AcX + cX -c^2 + A + 1) \Longrightarrow AX - Ac + c + \frac{A + 1 - c^2}{X}=\frac{N-1}{X^3}$
	\end{enumerate}
	In (1) we have $\frac{ab+1}{X} \in \mathbb{Z}$, and by assuming that $a \le b \Longrightarrow \frac{(X-a)(a)+1}{X} \in \mathbb{Z}$, therefore and as in \ref{lem33} (since $X > 11+4\sqrt{7}$) we can deduce that $a=1 \Longrightarrow \frac{N-1}{X^3}=1$ or $a=\sqrt{X+1} \Longrightarrow \frac{N-1}{X^3}=\sqrt{X+1}-1$ or $a > \sqrt{X}+2 \Longrightarrow \frac{(X-a)(a)+1}{X} = \frac{N-1}{X^3} > \sqrt{X}$ and the latter is impossible since $X \ge (N-1)^{2/7}$, hence, we have excluded all the cases of (1) (since we assumed that $N \neq p^{3(n-j)}+1$ and $N \neq (\sqrt{p^{n-j}+1}-1)p^{3(n-j)}+1$).
	\\
	In (2) we must have $\frac{A + 1 - c^2}{X} \in \mathbb{Z}$ which leads us to the following three cases
	\begin{enumerate}
		\item[($\alpha$)] $A + 1 - c^2 \ge X$
	\end{enumerate}
	which is impossible since $A < \sqrt{X}$.
	\begin{enumerate}
		\item[($\beta$)] $A+1-c^2=0 \Longrightarrow c = \sqrt{A+1}$
	\end{enumerate}
	In this case we will have $ab=(AX+\sqrt{A+1})(X-\sqrt{A+1})=AX^2-A \sqrt{A+1} X + \sqrt{A+1} X -(A+1)= AX(X-\sqrt{A+1}) + \sqrt{A+1}X-(A+1)$ which is clearly larger than $(X^{3/2})$, since $\sqrt{A+1}X-(A+1) > 0$ and $AX(X-\sqrt{A+1}) > X^{3/2}$ (since $ X-\sqrt{A+1} > \sqrt{X}$), hence, we have excluded this case as well.
	
	\begin{enumerate}
		\item[($\gamma$)] $c^2-A-1 \ge X$
	\end{enumerate}
	In this case we should have $A+1=c^2 \pmod X = (X-c)^2 \pmod X$, but since $(X-c) < \sqrt{X}$ then $(X-c)^2 < X$ and we can deduce that $A=(X-c)^2-1 \Longrightarrow \frac{N-1}{X^3}= ((X-c)^2-1)(X-c) + c -\frac{c^2 -1 -((X-c)^2-1)}{X} = (X-c)^3 -X +2c - \frac{c^2 -1 -X^2 +2cX -c^2 +1}{X}= (X-c)^3$, and we have excluded that case by the assumption $N \neq H^3 p^{3(n-j)}+1$. 
	\\
	Hence, we have excluded all the cases of (ii), and then we can deduce that $q=N$ and $N$ is prime.
\end{proof}

As an example to illustrate Theorem \ref{thm34}, consider $N=14\cdot 3^{18}+1=5423886847$. Then according to theorems \ref{thmjose} and \ref{thm2} we can only have $j \leq 7$, however taking $1 \le j \le 7$ and $a=1481700844$ will not satisfy (i) in neither \ref{thmjose} nor \ref{thm2}. Also in \ref{thm31} we can only have $j \leq 11$, and taking $1 \le j \le 11$ with $a=1481700844$ will not satisfy (i) in \ref{thm31} as well. But in \ref{thm34} we can have $j \leq 12$, and by taking $j=12$, we obtain $1481700844^{14\cdot 3^{18-12-1}}\equiv 3256260648 \pmod N$ and the result follows from Theorem \ref{thm34}. Note that the condition $N=Kp^n+1 \neq (\sqrt{p^{n-j}+1}-1)p^{3(n-j)}+1$ in \ref{thm34} is not necessary (it is not necessary in the following corollary as well), since the only cases we can have $p^{n-j}+1=\square$ is when $p^{n-j}=2^3$ or $3$ whose are smaller than $11+4 \sqrt{7}$. 

\begin{corollary}\label{cor35}
	Let $N=2Kp^3+1$,  where $p$ is prime, $p>11+4 \sqrt{7}$, $2K < \log_{a}(2Kp^3+1)$, $2K < \sqrt{p}$, and $2K$ is not a perfect cube. If there exists an integer $a>1$ such that:
	\begin{equation}
		a^{N-1} \equiv 1 \pmod N.
	\end{equation}
	Then $N$ is prime.
\end{corollary}

\begin{proof}
	Since $2K < \sqrt{p}$ then taking $j=2$ will satisfy the conditions $p^j \ge p^{2(3-j)} $ and $p^{3-j} \ge (N-1)^{2/7}$, and since $2K < \log_{a}(2Kp^3+1)$ then we will have $a^{2Kp^{n-j-1}}=a^{2Kp^{3-2-1}}=a^{2K}  \equiv L \neq 1 \pmod N$, also we have $a^{N-1} \equiv 1 \pmod N$ and $2K$ is not a perfect cube. Hence all the conditions of Theorem \ref{thm34} are satisfied and $N$ is prime.
\end{proof}

\section{A generalization of Sophie Germain and Safe primes} 

Recall that a prime $p$ is called a \emph{SafePrime} if $\frac{p-1}{2}$ is also prime. A prime $q$ is called a \emph{SophieGermainPrime} if $2q+1$ is prime. Thus there is a one-to-one correspondence between SafePrimes and SophieGermainPrimes. SafePrimes are useful cryptographic parameters as they are resilient against certain attacks, see \cite{Crypto, Rib} for more details.
In this section we provide certain classes of primes which generalize those notions. Our motivation comes from the observation that the key feature of SafePrimes is that $p-1$ has a ``large" prime factor (namely $\frac{p-1}{2}$). We thus propose the following extension. 


\begin{definition}
	$a$-SafePrime is a prime number of the form $N=2Kp+1$, where $p$ is prime and $2K < \log_{a}(2Kp+1)$
\end{definition}

Our purpose of this notion is to have the factor $p$ of $N-1$ be ``much larger" than any other factor; in particular it is   ``much larger" than  $\frac{N-1}{p}$. In the classical definition of SafePrimes we simply have $K=1$ and that condition is clearly satisfied. Note that the primality of $a$-SafePrimes can be certified only by proving that  $a^{N-1} \equiv 1 \pmod N$ (according to \ref{cor24}).

Now, we will define $a$-SophieGermainPrimes as those primes $p$ for which $2Kp+1$ is also prime for any $K$ with $2K < \log_{a}(2Kp+1) \Longrightarrow 2K \leq \log_{a}(2Kp) \Longrightarrow \frac{a^{2K}}{2K} \leq p $. 

Questions regarding the infinitude and density of $a$-SophieGermainPrimes seem to be more or less of the same order of difficulty as the corresponding questions for the classical case, where one mostly relies on conjectures and heuristics  (see \cite{Chris, Korevaar} for instance), but the distribution of $a$-SophieGermainPrimes will be addressed in the following section, and we will give an important conjecture about the distribution of $2$-SophieGermainPrimes (and $a$-SophieGermainPrimes in general). As for the density of  $a$-SafePrimes we recall Hardy and Littlewood's Conjecture B \cite{HarLit} which states that for $K\geq 1$,  the number of the prime pairs $(p, p+2K)$, where $p \leq x$, is asymptotically given by

\begin{equation}
	\label{asymp1}
	2C_K \frac{x}{\log^2 (x)} \sim 2C_K \int_{2}^{x} \frac{dt}{\log^2 t},
\end{equation}
\\
where

\[
\displaystyle{C_K=c_{2}\prod_{2<q, q\mid 2K}^{}\frac{q-1}{q-2}}.
\]

and $c_2$ (the so-called \emph{twin prime constant}) is given by

\[
c_2=\prod_{p\geq 3} \frac{p(p-2)}{(p-1)^2}\approx 0.66016181.
\]

The exact same asymptotics can be used to estimate  the number of the primes of the form $2Kp+1$, where $p$ is prime and $p \leq x$. Furthermore (as discussed in \cite{Chris} for instance), the following function, which is asymptotic to the right hand side of \eqref{asymp1} gives a more accurate estimate, especially for smaller values of $x$

\begin{equation}
	\label{asymp2}        
	2C_K \int_{2}^{x} \frac{dt}{\log t \log 2Kt}.
\end{equation}
\\
If we assume that the number of primes of the form $2Kp+1$ with $K$ ranging between $1$ and $n$ and $p\leq x$ is simply asymptotic to the sum of the individual estimates (as given in \eqref{asymp1} or \eqref{asymp2}) then we get the estimates:

\begin{equation}
	\label{asympSum1}
	2\frac{x}{\log^2(x)}(C_1+C_{2} +\dots+ C_{n})
\end{equation}

For which prime $p$ in the range  $\frac{a^{2n}}{2n} \leq p < \frac{a^{2(n+1)}}{2(n+1)} $ $K$ can take values up to $n$ with the condition $\frac{a^{2K}}{2K} \leq p$ satisfied. For typographical convenience we set $f(a,n):=\frac{a^{2n}}{2n}$. Utilizing \eqref{asympSum1} the asymptotic behavior for the number of $a$-SafePrimes for which $f(a,n) \leq p < f(a,n+1)$ is:

\[
\alpha_1(a,n):=2c_2\left(\frac{f(a,{n+1})}{\log^2 (f(a,n+1)) }- \frac{f(a,n)}{\log^2 (f(a,n)) } \right) \sum_{K=1}^{n} \prod_{2<q,q\mid 2K} \frac{q-1}{q-2} .
\]

Similarly, we can use the integral estimate in \eqref{asymp2} to obtain the estimate:

\[ 
\displaystyle{ \alpha_2(a,n) :=2c_2 \sum_{K=1}^{n} \prod_{2<q, q \mid 2K} \frac{q-1}{q-2} \int_{f(a,n)}^{f(a,n+1)} \frac{dt}{\log t \log 2Kt}}. 
\]

Table \ref{gsp} lists  samples of the actual count of $2$-SafePrimes for which $p$ is in the interval $[f(2,n), f(2,n+1)]$ (labeled $2$-SP for brevity), as well as the corresponding estimates $\alpha_1(2,n), \alpha_2(2,n)$ for $n$ up to $14$. We also list the number of $2$-SophieGermainPrimes in those intervals (labeled $2$-SGP).  
We can not say that there is a one-to-one correspondence between $2$-SP and $2$-SGP due to the simple fact that the same $2$-SGP might give rise to more than one $2$-SP for various values of $K$ (the same can be said as for correspondence between $a$-SP and $a$-SGP). For example when $n=3$ in the table below, the interval $[f(2,3), f(2,4)]=[2^6/6, 32]$ contains exactly $7$ primes; namely $11$ through $31$. Two of them ($19$ and $31$) are not $2$-SGP, but three others give rise to two $2$-SP each (namely $11$, $13$ and $23$).

\begin{table}[h]
	\centering
	
	\begin{tabularx}{\linewidth}{|c|X|X|X|X|}
		\hline
		$n$  & $2$-SGP in  $[f(2,n), f(2,n+1)]$ & $2$-SP $2Kp+1$, $p$ in $[f(2,n), f(2,n+1)]$ & $\alpha_1(2,n)$ & $\alpha_2(2,n)$ \\
		\hline
		1 & 2 & 2 & $-2.74$ & 1.5 \\
		\hline 
		2 & 2 & 2 & $-0.469$ & 3 \\
		\hline
		3 & 5 & 8 & 4 & 9 \\
		\hline
		4 & 14 & 18 & 14 & 20 \\
		\hline
		5 & 36 & 52 & 44 & 54 \\
		\hline 
		6 & 104 & 168 & 148 & 165 \\ 
		\hline
		7 & 295 & 463 & 450 & 483 \\
		\hline
		8 & 895 & 1414 & 1380 & 1447 \\
		\hline
		9 & 2970 & 4854 & 4724 & 4836 \\
		\hline
		10 & 9496 & 15783 & 15484 & 15672 \\
		\hline
		11 & 30788 & 50832 & 50827 & 51030 \\
		\hline
		12 & 104997 & 177808 & 178920 & 178090 \\
		\hline
		13 & 353357 & 596973 & 602484 & 597141 \\
		\hline
		14 & 1211233 & 2041459 & 2066125 & 2040547 \\
		\hline
	\end{tabularx}
	\caption{Some counts and asymptotics of $2$-SafePrimes} \label{gsp}       
\end{table}

\section{The probability that a given prime is $a$-SGP}

From the definition of $a$-SGP in the previous section, it is obvious that for which prime $f(a,n) \leq p \leq f(a,n+1)$, there is a set of numbers of the form $2Kp+1$, where $1 \leq K \leq n$, that if there is at least one prime of them, then we call $p$ an $a$-SGP, that set of numbers will be denoted by $S_p(n)$, namely, $S_p(n)=\{2p+1, 4p+1, 6p+1, ...... , 2np+1\}$. \\  
Now, we can't declare that the elements of $S_p(n)$ are independent, since if $2p+1$ is divisible by $3$ (for example), then all elements in $\{8p+1, 14p+1, 20p+1....\}$ will be divisible by $3$ as well, namely, we can't declare an element of $S_p(n)$ is independent of the remained elements of $S_p(n)$ unless it is not divisible by any prime less than $n$. Therefore, to get an approximation to the number of independent elements in $S_p(n)$ we should get an approximation to the number of elements in $S_p(n)$ which are coprime to all primes less than $n$ (except for 2, since $S_p(n)$ elements are coprime to 2 already), but since for any $N \in S_p(n)$ the probability that $(N,q)=1$ (such that $3 \le q$, and $q$ is prime) is $\frac{q-1}{q}$, hence, the probability that $(N,q)=1$ for all primes $3 \le q \le n$ is $\displaystyle{\prod_{3\leq q \leq n}^{} \frac{q-1}{q}}$, therefore, the number of elements in $S_p(n)$ which are coprime to all primes less than $n$ (it is also the number of independent elements in $S_p(n)$) is approximately     

\[
\lvert S_p(n) \rvert \prod_{3\leq q \leq n}^{} \frac{q-1}{q} = n \prod_{3\leq q \leq n}^{} \frac{q-1}{q} \approx n \frac{4e^{-\gamma}}{\log n^2}. 
\]

That set of independent numbers will be denoted by $ind$-$S_p(n)$. Now assume that $N=2Kp+1 \in ind$-$S_p(n)$, then the probability that $N$ is prime (since $N$ is in the interval $[2f(a,n) , 2nf(a,n+1)]$ and not divisible by any prime less than $n$) is about:

\[
\frac{\int_{2f(a,n)}^{2nf(a,n+1)} \frac{dt}{\log t}}{ (2nf(a,n+1)-2f(a,n)) \displaystyle{\prod_{2\leq q \leq n}} \frac{q-1}{q}}
\]

hence, the probability that $N$ is composite is about:

\[
1-\frac{\int_{2f(a,n)}^{2nf(a,n+1)} \frac{dt}{\log t}}{\frac{2e^{-\gamma}}{\log n^2} (2nf(a,n+1)-2f(a,n))}
\]

thus, the probability that all the elements of $ind$-$S_p(n)$ are composite is about:

\[
{\left(1-\frac{\int_{2f(a,n)}^{2nf(a,n+1)} \frac{dt}{\log t}}{\frac{2e^{-\gamma}}{\log n^2} (2nf(a,n+1)-2f(a,n))}\right)}^{ n \frac{4e^{-\gamma}}{\log n^2}}
\]

therefore, the probability that not all the elements of $ind$-$S_p(n)$ are composite is about:

\[
\alpha_3(a,n):=1-{\left(1-\frac{\int_{2f(a,n)}^{2nf(a,n+1)} \frac{dt}{\log t}}{\frac{2e^{-\gamma}}{\log n^2} (2nf(a,n+1)-2f(a,n))}\right)}^{ n \frac{4e^{-\gamma}}{\log n^2}}
\]

consequently, $\alpha_3(a,n)$ finds the probability that a given prime $f(a,n) \leq p \leq f(a,n+1)$ is $a$-SGP. 
In correspondence to $\alpha_3(a,n)$ we obtained a similar estimate by assuming that the elements of $S_p(n)$ are independent (in fact they are not), and surprisingly, this obtained estimate behaved similarly to $\alpha_3(a,n)$: 

\[
\alpha_4(a,n):=1-{\left(1-\frac{\int_{2f(a,n)}^{2nf(a,n+1)} \frac{dt}{\log t}}{\frac{1}{2} (2nf(a,n+1)-2f(a,n))}\right)}^{n}.
\]

But we should acknowledge that both of $\alpha_3(a,n)$ and $\alpha_4(a,n)$ failed to expect the number of $a$-SGP for large values of $a$ with small $n$, which compelled us to obtain the following estimate which depends on Hardy and Littlewood's Conjecture:

\[
\alpha_5(a,n):=1- \displaystyle{\prod_{K=1}^{n}} \left(1-\frac{2C_K\int_{f(a,n)}^{f(a,n+1)} \frac{dt}{\log t \log 2Kt}}{\int_{f(a,n)}^{f(a,n+1)} \frac{dt}{\log t}}\right)
\]

note that if we multiply either $\alpha_3(a,n)$, $\alpha_4(a,n)$ or $\alpha_5(a,n)$  by the expected number of primes in the interval $[f(a,n) , f(a,n+1)]$ (namely, $\int_{f(a,n)}^{f(a,n+1)} \frac{dt}{\log t}$), then we will get an approximation to the number of $a$-SGP primes in the interval $[f(a,n) , f(a,n+1)]$.
\\
Table \ref{estimates1} lists some counts and asymptotics of $2$-SGP primes in many different intervals which shows that the accuracy of either $\alpha_3(2,n)$, $\alpha_4(2,n)$ or $\alpha_5(2,n)$ is very acceptable.\\
But we should acknowledge that in some similar problems the data agrees well with heuristics only when the numbers are small, therefore we tried to test the accuracy of $\alpha_3(2,n)$, $\alpha_4(2,n)$ and $\alpha_5(2,n)$ when dealing with larger numbers, and surprisingly, their accuracy were very acceptable even if we are dealing with $100$, $1000$ or $10000$ digits numbers, check table \ref{estimates2}.

\begin{table}[!h]
	\centering
	\begin{tabularx}{\linewidth}{|c|X|X|X|X|X|}
		\hline
		n & $\pi(f(2,n+1))-\pi(f(2,n))$ & $2$-SGP in the interval $[f(2,n), f(2,n+1)]$ & $\alpha_3(2,n)$ $\int_{f(2,n)}^{f(2,n+1)} \frac{dt}{\log t}$ & $\alpha_4(2,n)$ $\int_{f(2,n)}^{f(2,n+1)} \frac{dt}{\log t}$ & $\alpha_5(2,n)$ $\int_{f(2,n)}^{f(2,n+1)} \frac{dt}{\log t}$ \\
		\hline
		1 & 2 & 2 & -- & 2.2 & 2.5 \\
		\hline
		2 & 2 & 2 & 2.8 & 3 & 2.9 \\
		\hline
		3 & 7 & 5 & 5.9 & 5.9 & 6.3 \\
		\hline
		4 & 15 & 14 & 14.1 & 13.8 & 13.9 \\
		\hline
		5 & 42 & 36 & 37 & 36 & 36 \\
		\hline
		6 & 124 & 104 & 104 & 100 & 102 \\
		\hline
		7 & 372 & 295 & 307 & 296 & 297 \\
		\hline
		8 & 1144 & 895 & 941 & 906 & 892 \\
		\hline
		9 & 3647 & 2970 & 2974 & 2864 & 2874 \\
		\hline
		10 & 11861 & 9496 & 9621 & 9272 & 9284 \\
		\hline
		11 & 39258 & 30788 & 31741 & 30617 & 30197 \\
		\hline
		12 & 132119 & 104997 & 106446 & 102770 & 102823 \\
		\hline
		13 & 450453 & 353357 & 361946 & 349778 & 346471 \\
		\hline
		14 & 1553274 & 1211233 & 1245390 & 1204660 & 1185604 \\
		\hline
		15 & 5411233 & 4304679 & 4329368 & 4191628 & 4233170 \\
		\hline
		16 & 19015798 & 14953724 & 15185847 & 14715731 & 14712906 \\
		\hline 
		17 & 67343702 & 52508562 & 53689079 & 52071171 & 51643313 \\
		\hline
		18 & 240139092 & 188600098 & 191152371 & 185542307 & 185655903 \\
		\hline
		19 & 861585192 & 671209186 & 684850472 & 665264341 & 660813007 \\
		\hline
	\end{tabularx}
	\caption{Some counts and asymptotics of $2$-SGP}
	\label{estimates1}
\end{table} 

\begin{table}[!h]
	
	\begin{tabularx}{\textwidth}{|c|X|X|c|c|c|}
		\hline
		$n$ & $L$ consecutive primes in $[f(2,n),f(2,n+1)]$ & $2$-SGP out of $L$ consecutive primes in $[f(2,n),f(2,n+1)]$ & $L \alpha_3(2,n)$ & $L \alpha_4(2,n)$ & $L \alpha_5(2,n)$ \\
		\hline
		$169$ & $100000$ & 76832 & 76986 & 76460 & 76346 \\
		\hline
		$1666$ & $10000$ & 7669 & 7646 & 7638 & 7636 \\
		\hline
		$16617$ & $100$ & 77 & 76 & 76 & 76 \\
		\hline    
	\end{tabularx}
	\caption{Some counts and asymptotics of $2$-SGP out of some $100$, $1000$, and $10000$ digits consecutive primes}
	\label{estimates2}
\end{table}

The computations in tables \ref{estimates1} and \ref{estimates2} show that the ratio of $2$-SGP to primes slowly converges to $0.76$, also our computations of $\alpha_3(2,n)$, $\alpha_4(2,n)$ and $\alpha_5(2,n)$ for some large values of $n$ show that they converge to $0.76$ which means that the probability that a given prime $p$ is $2$-SGP converges to $0.76$ as well (as $p \to \infty$). Therefore, according to our computations of $\alpha_3(2,n)$, $\alpha_4(2,n)$  and $\alpha_5(2,n)$ and supported by the computations in tables \ref{estimates1} and \ref{estimates2}, we believe that it is reasonable to give the following conjecture 

\begin{conjecture}\label{con51}
	The probability that a given prime $p$ is $2$-SGP converges to $0.76$ as $p \to \infty$.
\end{conjecture}

We also computed $\alpha_3(a,n)$, $\alpha_4(a,n)$ and $\alpha_5(a,n)$ for some large values of $n$ with many values of $a$ (other than $a=2$), and according to these computations and supported by the computations in table \ref{estimates3} we give a general conjecture about the probability that a given prime $p$ is $a$-SGP

\begin{conjecture}\label{con52}
	The probability that a given prime $p$ is $a$-SGP converges to $\displaystyle{\lim_{n \to \infty} \alpha_3(a,n) \approx }$ $ \displaystyle{\lim_{n \to \infty} \alpha_4(a,n) } \approx \lim_{n \to \infty} \alpha_5(a,n) $ as $p \to \infty$.
\end{conjecture}

\begin{table}[!h]
	
	\begin{tabularx}{\textwidth}{|c|X|X|c|c|c|}
		\hline
		$(a,n)$ & $L$ consecutive primes in $[f(a,n),f(a,n+1)]$ & $a$-SGP out of $L$ consecutive primes in $[f(a,n),f(a,n+1)]$ & $L \alpha_3(a,n)$ & $L \alpha_4(a,n)$ & $L \alpha_5(a,n)$ \\
		\hline
		$(2,169)$ & $100000$ & 76832 & 76986 & 76460 & 76346 \\
		\hline
		$(3,106)$ & $100000$ & 59943 & 60230 & 59727 & 59555 \\
		\hline
		$(4,84)$ & $100000$ & 51493 & 51724 & 51276 & 50998 \\
		\hline
		$(5,72)$ & $100000$ & 46186 & 46512 & 46104 & 45798 \\
		\hline
		$(6,65)$ & $100000$ & 42743 & 42932 & 42557 & 42077 \\
		\hline
		$(7,59)$ & $100000$ & 39449 & 40288 & 39934 & 39239 \\
		\hline 
		$(8,56)$ & $100000$ & 37924 & 38236 & 37906 & 37316 \\
		\hline 
		$(9,53)$ & $100000$ & 36224 & 36585 & 36271 & 35606 \\
		\hline
		$(10,50)$ & $100000$ & 34702 & 35217 & 34915 & 34332 \\
		\hline       
	\end{tabularx}
	\caption{Some counts and asymptotics of $a$-SGP out of $10^5$ $100$ digits consecutive primes}
	\label{estimates3}
\end{table}

To show the utility of \ref{con51} we considered the question of which of the Mersenne primes (i.e. those of the form $2^p-1$ where $p$ itself is prime) are also $2$-SGP. Either $\alpha_3(2,n)$, $\alpha_4(2,n)$ or $\alpha_5(2,n)$ (or according to \ref{con51} ) can tell us that there are about 23 $2$-SGP out of the first 30 Mersenne primes, if that was true then we will be able to generate 23 $2$-SP at least. But, out of the first 30 Mersenne primes there are exactly 21 $2$-SGP, which shows again that the accuracy of either $\alpha_3(2,n)$, $\alpha_4(2,n)$ or $\alpha_5(2,n)$ (or \ref{con51}) is acceptable (even if we are dealing with completely random numbers such as Mersenne primes) and we can use them as a method to search for huge $2$-SP or $a$-SP in general. The following table lists all of the 21 Mersenne primes which are also $2$-SGP as well as the corresponding values of $2K$ for each one (for some of them, more than one value of $K$ works). 

\begin{table}[!h]
	\centering
	\begin{tabular}{|l|l|l|l|}
		\hline
		$N=2K \times (2^p-1)+1$ & \textrm{digits} & $N=2K \times (2^p-1)+1$ & \textrm{digits} \\
		\hline
		$2 \times (2^{2}-1)+1$ & 1  & $2010 \times (2^{4253}-1)+1$ & 1284 \\
		\hline
		$4 \times (2^{3}-1)+1$ & 2  & $3708 \times (2^{4253}-1)+1$ & 1284 \\
		\hline
		$4 \times (2^{7}-1)+1$ & 3  & $1746 \times (2^{9689}-1)+1$ & 2920 \\
		\hline
		$12 \times (2^{17}-1)+1$ & 7  & $3426 \times (2^{9941}-1)+1$ & 2997 \\
		\hline
		$16 \times (2^{19}-1)+1$ & 7 & $3696 \times (2^{9941}-1)+1$ & 2997 \\
		\hline
		$52 \times (2^{61}-1)+1$ & 21  & $6238 \times (2^{11213}-1)+1$ & 3380 \\
		\hline
		$66 \times (2^{61}-1)+1$ & 21  & $9048 \times (2^{11213}-1)+1$ & 3380 \\
		\hline
		$114 \times (2^{127}-1)+1$ & 41  & $858 \times (2^{21701}-1)+1$ & 6536 \\
		\hline
		$124 \times (2^{127}-1)+1$ & 41  & $14712 \times (2^{21701}-1)+1$ & 6537 \\
		\hline
		$336 \times (2^{521}-1)+1$ & 160  & $4018 \times (2^{23209}-1)+1$ & 6991 \\
		\hline
		$154 \times (2^{607}-1)+1$ & 185  & $20808 \times (2^{23209}-1)+1$ & 6991 \\
		\hline
		$550 \times (2^{607}-1)+1$ & 186  & $17262 \times (2^{44497}-1)+1$ & 13400 \\
		\hline
		$156 \times (2^{2203}-1)+1$ & 666  & $15418 \times (2^{86243}-1)+1$ & 25966 \\
		\hline
		$546 \times (2^{2203}-1)+1$ & 666  & $42844 \times (2^{86243}-1)+1$ & 25967 \\
		\hline
		$1110 \times (2^{2203}-1)+1$ & 667  & $58818 \times (2^{86243}-1)+1$ & 25967 \\
		\hline
		$1144 \times (2^{2203}-1)+1$ & 667  & 6526 $\times (2^{132049}-1)+1$ & 39755 \\
		\hline
		$1086 \times (2^{2281}-1)+1$ & 690  & 30690 $\times (2^{132049}-1)+1$ & 39756 \\
		\hline
		$1656 \times (2^{2281}-1)+1$ & 690 & 47142 $\times (2^{132049}-1)+1$ & 39756 \\
		\hline
		$1816 \times (2^{3217}-1)+1$ & 972 & 110086 $\times (2^{132049}-1)+1$  & 39756 \\
		\hline
	\end{tabular}
	\caption{Mersenne primes which are also $2$-SophieGermainPrimes, and the corresponding values of $K$} \label{mers}
\end{table}

We also considered the question of what is the probability that the largest known prime $p=2^{82589933}-1$ is $2$-SGP, that probability can be found by computing either $\alpha_3(2,n)$, $\alpha_4(2,n)$ or $\alpha_5(2,n)$ for $n=41294979$, hence, the probability that $p=2^{82589933}-1$ is $2$-SGP is about $\alpha_3(2,41294979) \approx \alpha_4(2,41294979) \approx \alpha_5(2,41294979) \approx 0.7637$, which means that the probability to find a prime number of the form $N=2K(2^{82589933}-1)+1$, where $1 \leq K \leq 41294979$ is about $0.7637$ .

\section*{Acknowledgment}
I am grateful to my supervisors Prof. Ahmed El-Guindy of Cairo University and Prof. Hatem M.Bahig of Ain Shams University for all their help and guidance that they have given me. Also I would like to thank Prof. Wafik Lotfallah of the American University in Cairo and Ms. Hagar Gamal of Cairo University for several useful discussions.

\end{document}